\documentclass[dvipdfmx,a4paper,12pt]{amsart}
\usepackage{amsmath,amssymb,amsthm}
\usepackage{tikz-cd}
\usepackage{mathtools}
 \mathtoolsset{showonlyrefs=true}
\usepackage[colorlinks=true]{hyperref}

\allowdisplaybreaks

\numberwithin{equation}{section}
 \theoremstyle{definition}
  \newtheorem{definition}[equation]{Definition}
  \newtheorem*{definition*}{Definition}
  
  \newtheorem*{example*}{Example}
 \theoremstyle{plain}
  
  \newtheorem*{lemma*}{Lemma}
  \newtheorem{proposition}[equation]{Proposition}
  \newtheorem*{proposition*}{Proposition}
  \newtheorem{theorem}[equation]{Theorem}
  \newtheorem*{theorem*}{Theorem}
  
  \newtheorem*{corollary*}{Corollary}
 \theoremstyle{remark}
  \newtheorem{remark}[equation]{Remark}
  \newtheorem*{remark*}{Remark}

  \newcommand{\Homeo}{{\mathrm{Homeo}}}
  \newcommand{\abs}[1]{\left\lvert{#1}\right\rvert}
  \newcommand{\floor}[1]{\left\lfloor{#1}\right\rfloor}

\begin{document}
\title[bounded Euler class and symplectic rotation number]{The Bounded Euler Class and the Symplectic Rotation Number}
\author{Daiki UDA}
\address{Graduate School of Mathematics, Nagoya University, Japan}
\email{m19007g@math.nagoya-u.ac.jp}

\begin{abstract}
  Ghys~\cite{Ghys01} established the relationship between the bounded Euler class in $H_{b}^{2}(\Homeo_{+}(S^{1});\mathbb{Z})$ and the Poincar\'{e} rotation number, that is,
  he proved that the pullback of the bounded Euler class under a homomorphism $\varphi \colon \mathbb{Z} \to \Homeo_{+}(S^{1})$ coincides with the Poincar\'{e} rotation number of $\varphi(1)$.
  In this paper, we extend the above result to the symplectic group in some sense,
  and clarify the relationship between the bounded Euler class in $H_{b}^{2}(Sp(2n;\mathbb{R});\mathbb{Z})$ and the symplectic rotation number investigated in Barge-Ghys~\cite{B-G}.
\end{abstract}

\maketitle

\section{Introduction and Statement of Result}
Let $G = \Homeo_{+}(S^{1})$ be the group of orientation-preserving homeomorphisms on the circle $S^{1}$ and $\widetilde{G}$ denote its universal cover.
It is known that $\widetilde{G}$ is identified with the group of self-homeomorphisms of $\mathbb{R}$ which commute with translations by integers:
\[\widetilde{G} \cong \{ \tilde{g} \in \Homeo_{+}(\mathbb{R}) \mid \forall x \in \mathbb{R}, \tilde{g}(x + 1) = \tilde{g}(x) + 1 \}.\]
It is also known that there exist mappings $\tau \colon \widetilde{G} \to \mathbb{R}$ and $\rho \colon G \to \mathbb{R}/\mathbb{Z}$
called, respectively, the Poincar\'{e} translation number and the Poincar\'{e} rotation number,
which are defined by the formulae:
\[\tau(\tilde{g}) = \lim_{n \to \infty}\frac{\tilde{g}^{n}(0)}{n};\ \rho(g) = \tau(\tilde{g}) \bmod \mathbb{Z},\]
where $\tilde{g}$ is any lift of $g$ in the second formula.

Let $H^{\bullet}(\Gamma;\mathbb{Z})$ (resp. $H_{b}^{\bullet}(\Gamma;\mathbb{Z})$) denote the group cohomology (resp. the bounded cohomology) of a discrete group $\Gamma$ with coefficients in $\mathbb{Z}$ (see Section~\ref{Preliminaries}.).
For $\Gamma = \mathbb{Z}$, it is known that a homomorphism $\mathbb{R} \to H_{b}^{2}(\mathbb{Z};\mathbb{Z}); r \mapsto -[c_{r}]$
induces an isomorphism $\mathbb{R}/\mathbb{Z} \cong H_{b}^{2}(\mathbb{Z};\mathbb{Z})$.
Here the bounded $2$-cocycle $c_{r} \colon \mathbb{Z} \times \mathbb{Z} \to \mathbb{Z}$ is given by
\[c_{r}(n, m) = \floor{rn} + \floor{rm} - \floor{r(n+m)},\]
where $\floor{\cdot}$ denotes the floor function.
On the other hand, for $\Gamma = G$, it is known that the universal cover $\widetilde{G} \to G$ induces a central $\mathbb{Z}$-extension
\[0 \to \mathbb{Z} \to \widetilde{G} \to G \to 1,\]
and this central extension determines the Euler class $e \in H^{2}(G; \mathbb{Z})$.
Furthermore, the Euler class $e$ has a bounded representative (see, e.g., \cite[Lemma 6.3]{Ghys01})
and the resulting class is called the bounded Euler class $e_{b} \in H_{b}^{2}(G;\mathbb{Z})$.

Ghys investigated these objects, and found the following relationship between the bounded Euler class $e_{b}$ and the Poincar\'{e} rotation number:
\begin{theorem}[{\cite[Theorem 6.4]{Ghys01}}]
  \label{Ghys}
  Given $g \in G$, let $\varphi_{g} \colon \mathbb{Z} \to G$ be a homomorphism defined by $\varphi_{g}(k) = g^{k}$.
  Then, under the identification $H_{b}^{2}(\mathbb{Z};\mathbb{Z}) \cong \mathbb{R}/\mathbb{Z}$, the pullback $\varphi_{g}^{*}(e_{b})$ of the bounded Euler class coincides with the Poincar\'{e} rotation number of $g$. \qed
\end{theorem}
For more details, see Ghys' survey article \cite{Ghys01}.

In the paper Barge-Ghys~\cite{B-G},
they defined the symplectic rotation number and the symplectic translation number on the symplectic group $Sp(2n;\mathbb{R})$ and its universal cover, respectively,
and gave the way to compute the symplectic rotation number for a symplectic matrix in terms of its eigenvalues.

One can easily see that the definitions of these maps are similar to those of Poincar\'{e}'s.
Furthermore, as in the case of $\Homeo_{+}(S^{1})$, the universal cover of $Sp(2n; \mathbb{R})$ induces a central $\mathbb{Z}$-extension,
which corresponds to the bounded Euler class in $H_{b}^{2}(Sp(2n;\mathbb{R}); \mathbb{Z})$.
In this paper, we clarify the relationship between the bounded Euler class and the symplectic rotation number.
The following is our main result, which is comparable to Theorem~\ref{Ghys} due to Ghys:
\begin{theorem}
  \label{main}
  Given a symplectic matrix $g \in Sp(2n;\mathbb{R})$, let $\varphi_{g} \colon \mathbb{Z} \to Sp(2n;\mathbb{R})$ be a homomorphism defined by $\varphi_{g}(k) = g^{k}$.
  Then, under the identification $H_{b}^{2}(\mathbb{Z};\mathbb{Z}) \cong \mathbb{R}/\mathbb{Z}$,
  the pullback of the bounded Euler class by $\varphi_{g}$ coincides with the symplectic rotation number of $g$.\qed
\end{theorem}

\section{Preliminaries}
\label{Preliminaries}
In this section, we review the definitions and properties of some concepts which we will need to state and prove our result.
Let $G$ be a group and $A$ an abelian group.
For each integer $k \geq 0$, let $C^{k}(G,A)$ denote the set of all mappings of $G^{k}$ into A, and define $\delta \colon C^{k}(G, A) \to C^{k+1}(G, A)$ by
\[\begin{gathered}
  \delta c(g_{1}, \ldots, g_{k+1}) = c(g_{2}, \ldots, g_{k+1}) + \sum_{i = 1}^{k} (-1)^{i}c(g_{1}, \ldots, g_{i}g_{i+1}, \ldots, g_{k+1}) \\
  \text{} + (-1)^{k+1} c(g_{1}, \ldots, g_{k}).
\end{gathered}\]
Then $(C^{\bullet}(G,A), \delta)$ constitutes a cochain complex and its cohomology $H^{\bullet}(G;A)$ is called the \emph{group cohomology of $G$ with coefficients in $A$}.
It is known that the second group cohomology $H^{2}(G;A)$ and the set of equivalence classes of central $A$-extensions of $G$ are in $1$-$1$ correspondence:
Here a \emph{central $A$-extension of $G$} is a short exact sequence
\[0 \longrightarrow A \stackrel{i}{\longrightarrow} \Gamma \stackrel{p}{\longrightarrow} G \longrightarrow 1\]
of groups with $i(A) \subset Z(\Gamma)$, where $Z(\Gamma)$ is the center of $\Gamma$.
Sometimes we simply refer $\Gamma$ as a central $A$-extension of $G$.

The correspondence is given as follows:
For a central $A$-extension $\Gamma$ of $G$, choose a set-theoretic section $s \colon G \to \Gamma$ of a surjection $p \colon \Gamma \to G$.
Since $s(g_{1}g_{2})^{-1}s(g_{1})s(g_{2}) \in \ker(p) = i(A)$ for each $g_{1}, g_{2} \in G$, we can define a $2$-cochain $\chi \colon G^{2} \to A$ by the formula
\[i(\chi(g_{1}, g_{2})) = s(g_{1}g_{2})^{-1}s(g_{1})s(g_{2}).\]
Then it is straightforward to check that $\chi$ is indeed a $2$-cocycle and a cohomology class $[\chi] \in H^{2}(G;A)$ determined by $\chi$ is independent of the choice of a section.
This cohomology class is called the \emph{Euler class} of the extension and denoted by $e(\Gamma)$.
Assigning $e(\Gamma)$ to the equivalence class of $\Gamma$ is the required correspondence.

If $A$ has a norm (e.g., $A = \mathbb{Z}$ or $\mathbb{R}$ with the usual absolute value as its norm), we can consider a subcomplex $C_{b}^{\bullet}(G,A)$ of $C^{\bullet}(G,A)$ consisting of bounded cochains,
and its cohomology denoted by $H_{b}^{\bullet}(G;A)$ is called the \emph{bounded cohomology of G with coefficients in $A$}.
Given a central $A$-extension $\Gamma$ of $G$, if the $2$-cocycle $\chi \colon G^{2} \to A$ defined above is bounded,
then its cohomology class $[\chi] \in H_{b}^{2}(G;A)$ is called the \emph{bounded Euler class} of the extension $\Gamma$ and denoted by $e_{b}(\Gamma)$.

A $1$-cochain $f \in C^{1}(G,\mathbb{R})$ is called a \emph{quasi-morphism} if its coboundary is bounded, i.e.,
\[\sup_{g_{1}, g_{2} \in G} \abs{f(g_{1}) + f(g_{2}) - f(g_{1}g_{2})} < + \infty\]
holds.
A quasi-morphism $f \colon G \to \mathbb{R}$ is called a \emph{homogeneous quasi-morphism} if $f(g^{n}) = n \cdot f(g)$ for all $g \in G$ and $n \in \mathbb{Z}$.
There exists a way to construct a homogeneous quasi-morphism out of a quasi-morphism:
\begin{proposition}
  For a quasi-morphism $f \colon G \to \mathbb{R}$, define a map $\overline{f} \colon G \to \mathbb{R}$ by the formula
  \[\overline{f}(g) = \lim_{n \to \infty} \frac{f(g^{n})}{n}.\]
  Then this map is a unique homogeneous quasi-morphism on $G$ which satisfies the condition $\sup_{g \in G} \abs{\overline{f}(g) - f(g)} < + \infty$.
  The map $\overline{f}$ is called the \emph{homogenization} of $f$.\qed
\end{proposition}
Note that if $f \colon G \to \mathbb{Z}$ is a quasi-morphism, then its coboundary $\delta f$ is a bounded $2$-cocycle,
and hence defines a bounded cohomology class in $H_{b}^{2}(G;\mathbb{Z})$.
Using this remark, we can determine the second bounded cohomology of $\mathbb{Z}$ with coefficients in $\mathbb{Z}$:
\begin{proposition}
  \label{bdd-coh-of-Z}
  For a real number $r \in \mathbb{R}$, define a $1$-cochain $\beta_{r} \colon \mathbb{Z} \to \mathbb{Z}$ by $\beta_{r}(n) = \floor{rn}$,
  where $\floor{\cdot} \colon \mathbb{R} \to \mathbb{Z}$ is the floor function.
  Then $\beta_{r}$ is a quasi-morphism, and a homomorphism $\mathbb{R} \to H_{b}^{2}(\mathbb{Z};\mathbb{Z})$
  defined by $r \mapsto -[\delta\beta_{r}]$ induces an isomorphism $\mathbb{R}/\mathbb{Z} \cong H_{b}^{2}(\mathbb{Z};\mathbb{Z})$.\qed
\end{proposition}
Consult Frigerio~\cite{Frigerio} for proofs and further details.

Now recall that a connection cochain introduced in Moriyoshi~\cite{Moriyoshi}:
\begin{definition}
  Let $0 \to A \to \Gamma \to G \to 1$ be a central $A$-extension of $G$.
  A \emph{connection cochain} is a $1$-cochain $\tau \colon \Gamma \to A$ which satisfies the condition
  \[\tau(\gamma \cdot i(a)) = \tau(\gamma) + a\]
  for all $\gamma \in \Gamma, a \in A$.
\end{definition}
\begin{remark}
  Let $\Gamma$ be a central $A$-extension of $G$.
  Then there exists a $1$-$1$ correspondence between the sets of sections of $\Gamma \to G$ and the sets of connection cochains for this extension.
\end{remark}
\begin{remark}
  Given a central $A$-extension $\Gamma$ of G, let $B$ be an abelian group and $j \colon A \to B$ a homomorphism.
  Then a $1$-cochain $\tau \colon \Gamma \to B$ which satisfies the condition
  \[\tau(\gamma \cdot i(a)) = \tau(\gamma) + j(a)\]
  for all $\gamma \in \Gamma, a \in A$ is called a \emph{$B$-valued connection cochain}.
  There exists a way to construct a central $B$-extension $\Gamma_{B}$ of $G$, and then $\tau$ induces a connection cochain $\tau_{B} \colon \Gamma_{B} \to B$ for this extension.
\end{remark}
The fundamental property of connection cochain is the following
\begin{proposition}
  \label{conn-coch}
  Let $0 \to A \to \Gamma \to G \to 1$ be a central $A$-extension of $G$.
  For a connection cochain $\tau \colon \Gamma \to A$, there exists a unique $2$-cocycle $\sigma \colon G^{2} \to A$ such that $\delta \tau = p^{*} \sigma \colon \Gamma^{2} \to A$.
  Moreover the Euler class $e(\Gamma) \in H^{2}(G;A)$ coincides with $[-\sigma]$.\qed
\end{proposition}
Finally, we recall that a group $G$ is \emph{uniformly perfect} if there exists an integer $k$ such that any element of the group $G$ is a product of at most $k$ commutators.
The fundamental property of central extension of a uniformly perfect group is
\begin{proposition}[{\cite[p.236f.]{B-G}}]
  \label{fundamental-remark}
  Let $0 \to \mathbb{Z} \to \Gamma \to G \to 1$ be a central $\mathbb{Z}$-extension of a uniformly perfect group $G$.
  Then there exists at most one function $\tau \colon \Gamma \to \mathbb{R}$ which is a homogeneous quasi-morphism and, simultaneously, an $\mathbb{R}$-valued connection cochain.\qed
\end{proposition}
The groups $\Homeo_{+}(S^{1})$ and $Sp(2n;\mathbb{R})$ are both uniformly perfect, and then this proposition is essential to define the rotation numbers.

\section{Proof of Main Theorem}
Let $Sp(2n;\mathbb{R})$ be the group of symplectic matrices with respect to the standard symplectic form on $\mathbb{R}^{2n}$.
Since $\pi_{1}(Sp(2n;\mathbb{R})) \cong \mathbb{Z}$, its universal covering group $p \colon \widetilde{Sp}(2n;\mathbb{R}) \to Sp(2n;\mathbb{R})$ induces a central $\mathbb{Z}$-extension
\[0 \longrightarrow \mathbb{Z} \stackrel{i}{\longrightarrow} \widetilde{Sp}(2n;\mathbb{R}) \stackrel{p}{\longrightarrow} Sp(2n;\mathbb{R}) \longrightarrow 1.\]
There exist several quasi-morphisms on the group $\widetilde{Sp}(2n;\mathbb{R})$ which are also $\mathbb{R}$-valued connection cochains (\cite[C-1]{B-G}),
but according to Proposition~\ref{fundamental-remark}, their homogenization which we will denote $\tau \colon \widetilde{Sp}(2n;\mathbb{R}) \to \mathbb{R}$ is uniquely determined.
The map $\tau$ is called the \emph{symplectic translation number}.

Note that for each $g \in Sp(2n;\mathbb{R})$, the value $\tau(\tilde{g}) \bmod \mathbb{Z}$ is independent of the choice of a lift $\tilde{g} \in p^{-1}(g)$.
This is because $\tau$ is an $\mathbb{R}$-valued connection cochain.
Then define the \emph{symplectic rotation number} $\rho \colon Sp(2n;\mathbb{R}) \to \mathbb{R}/\mathbb{Z}$ by the formula
\[\rho(g) = \tau(\tilde{g}) \bmod \mathbb{Z},\]
where $\tilde{g} \in p^{-1}(g)$ is any lift of $g$.

Let $s \colon Sp(2n;\mathbb{R}) \to \widetilde{Sp}(2n;\mathbb{R})$ be a section of $p$ corresponding to a connection cochain $\floor{\cdot} \circ \tau \colon \widetilde{Sp}(2n;\mathbb{R}) \to \mathbb{Z}$.
Then, by Proposition~\ref{conn-coch}, there exists a $2$-cocycle $\sigma \colon Sp(2n;\mathbb{R}) \times Sp(2n;\mathbb{R}) \to \mathbb{Z}$ such that $p^{*} \sigma = \delta(\floor{\cdot} \circ \tau)$.
Since $\floor{\cdot} \circ \tau$ is a quasi-morphism, the cocycle $- \sigma$ is bounded,
and hence defines the bounded Euler class $e_{b} = [- \sigma] \in H_{b}^{2}(Sp(2n;\mathbb{R});\mathbb{Z})$.
Note that the homogenization of $\floor{\cdot} \circ \tau$ is nothing but the symplectic translation number $\tau$.
\begin{theorem}
  Given a symplectic matrix $g \in Sp(2n;\mathbb{R})$, let $\varphi_{g} \colon \mathbb{Z} \to Sp(2n;\mathbb{R})$ be a homomorphism defined by $\varphi_{g}(k) = g^{k}$.
  Then, under the identification $H_{b}^{2}(\mathbb{Z};\mathbb{Z}) \cong \mathbb{R}/\mathbb{Z}$,
  the pullback of the bounded Euler class $e_{b} \in H_{b}^{2}(Sp(2n;\mathbb{R});\mathbb{Z})$ by $\varphi_{g}$ coincides with the symplectic rotation number of $g$, that is,
  \[\varphi_{g}^{\,*}(e_{b}) = \rho(g)\]
  holds.
\end{theorem}
\begin{proof}
  For $g \in Sp(2n;\mathbb{R})$, define a quasi-morphism $\beta \colon \mathbb{Z} \to \mathbb{Z}$ by $\beta(k) = \floor{\tau(s(g)^{k})}$.
  Then we have
  \begin{align}
    \varphi_{g}^{\,*}\,\sigma (k, \ell)
    &= \sigma (g^{k}, g^{\ell}) = p^{*} \sigma (s(g)^{k}, s(g)^{\ell})\\
    &= \delta (\floor{\cdot} \circ \tau) (s(g)^{k}, s(g)^{\ell}) = \delta \beta (k, \ell),
  \end{align}
  for all $k, \ell \in \mathbb{Z}$.
  Let $\overline{\beta} \colon \mathbb{Z} \to \mathbb{R}$ be the homogenization of $\beta$, and set $r = \overline{\beta}(1) \in \mathbb{R}$:
  \[r = \overline{\beta}(1) = \lim_{k \to \infty} \frac{1}{k}\beta(k) = \lim_{k \to \infty} \frac{1}{k}\floor{\tau(s(g)^{k})} = \tau(s(g)).\]
  Since the symplectic translation number $\tau$ is a homogeneous quasi-morphism,
  \[\beta(k) = \floor{\tau(s(g)^{k})} = \floor{\tau(s(g)) \cdot k} = \beta_{r}(k),\]
  for all $k \in \mathbb{Z}$, and hence $\beta = \beta_{r}$.
  Therefore, under the identification $H_{b}^{2}(\mathbb{Z};\mathbb{Z}) \cong \mathbb{R}/\mathbb{Z}$ (Proposition~\ref{bdd-coh-of-Z}), we have
  \[\varphi_{g}^{\,*}(e_{b}) = \varphi_{g}^{\,*}([-\sigma]) = - [\delta \beta_{r}] = r \bmod \mathbb{Z} = \rho(g),\]
  as desired.
\end{proof}

\end{document}